\documentclass{amsart}
\usepackage{latexsym,amssymb,amsmath,amsthm,amscd,graphicx}

\makeatletter
\@namedef{subjclassname@2010}{%
  \textup{2010} Mathematics Subject Classification}
\makeatother

\numberwithin{equation}{section}
\newtheorem{theorem}{Theorem}[section]
\newtheorem{lemma}[theorem]{Lemma}

\newtheorem{proposition}[theorem]{Proposition}

\theoremstyle{definition}

\newtheorem{remark}[theorem]{Remark}
\newtheorem{definition}[theorem]{Definition}
\newtheorem{example}[theorem]{Example}

\theoremstyle{remark}

\newcommand{\cA}{{\mathcal A}}
\newcommand{\cB}{{\mathcal B}}
\newcommand{\cC}{{\mathcal C}}
\newcommand{\cD}{{\mathcal D}}

\newcommand{\cM}{{\mathcal M}}

\newcommand{\cQ}{{\mathcal Q}}

\newcommand{\cZ}{{\mathcal Z}}

\newcommand{\M}{{\mathbb M}}
\newcommand{\N}{{\mathbb N}}

\newcommand{\R}{{\mathbb R}}

\newcommand{\Z}{{\mathbb Z}}

\def\al{\alpha}

\def\eps{\varepsilon}

\def\lm{\lambda}
\def\Lm{\Lambda}
\def\sg{\sigma}

\def\0{\emptyset}
\def\6{\partial}
\def\8{\infty}

\def\l{\left}
\def\r{\right}

\def\wt{\widetilde}
\def\wht{\widehat}
\def\supp{\,{\rm supp}\,}
\def\sgn{\,{\rm sgn}\,}

\begin{document}

\title[The Fatou property of block spaces]{The Fatou property of block spaces}

\author[Y.~Sawano]{Yoshihiro Sawano}
\address{Department of Mathematics and Information Sciences, Tokyo Metropolitan University, 1-1 Minami Ohsawa, Hachioji, Tokyo 192-0397, Japan}
\email{ysawano@tmu.ac.jp}

\author[H.~Tanaka]{Hitoshi Tanaka}
\address{Graduate School of Mathematical Sciences, The University of Tokyo, Tokyo, 153-8914, Japan}
\email{htanaka@ms.u-tokyo.ac.jp}

\thanks{
The first author is supported by 
Grant-in-Aid for Young Scientists (B) (No.~24740085), 
the Japan Society for the Promotion of Science.
The second author is supported by 
the FMSP program at Graduate School of Mathematical Sciences, the University of Tokyo, 
and Grant-in-Aid for Scientific Research (C) (No.~23540187), 
the Japan Society for the Promotion of Science. 
}
\date{}

\begin{abstract}
Around thirty years ago, 
block spaces, 
which are the predual of Morrey spaces, 
had been considered. 
However, 
it seems that there is no proof that 
block spaces satisfy the Fatou property. 
In this paper 
the Fatou property for block spaces is verified and 
the predual of block spaces is characterized. 
\end{abstract}

\maketitle

\section{Introduction}\label{sec1}
The purpose of this paper is to verify the Fatou property for block spaces, 
which in turn yields a characterization for the predual of block spaces. 
The Morrey space $\cM^p_q(\R^n)$
is a properly wider space than 
the Lebesgue space $L^p(\R^n)$ 
when $0<q<p<\8$ and 
this space works well with the fractional integral operators 
(cf. \cite{SaSuTa1,SaSuTa2,SaSuTa3,TaGu}). 
We first recall the definition of Morrey spaces and 
the consider block spaces which are the predual of Morrey spaces. 

\subsection{Morrey spaces.} 
Let $0<q\le p<\8$ be two real parameters. 
For $f\in L^q_{{\rm loc}}(\R^n)$, define 
\begin{align*}
\|f\|_{\cM^p_q(\R^n)}
:=
\sup_{Q\in\cQ}
|Q|^{\frac1p-\frac1q}
\l(\int_{Q}|f(x)|^q\,dx\r)^{\frac1q}
=
\sup_{Q\in\cQ}
|Q|^{\frac1p}
\l(\frac1{|Q|}\int_{Q}|f(x)|^q\,dx\r)^{\frac1q},
\end{align*}
where we have used the notation 
$\cQ$ to denote the family of all cubes in $\R^n$ 
with sides parallel to the coordinate axes and 
$|Q|$ to denote the volume of $Q$. 
The Morrey space $\cM^p_q(\R^n)$ 
is defined to be the subset of all $L^q$ locally integrable functions $f$ on $\R^n$ for which 
$\|f\|_{\cM^p_q(\R^n)}$
is finite. 
It is easy see that 
$\|\cdot\|_{\cM^p_q(\R^n)}$
becomes the norm if $q\ge 1$ 
and that becomes the quasi norm otherwise. 
Letting
$0<r<q\le p<\8$ and 
using H\"{o}lder's inequality, 
we have
\begin{equation}\label{1.1}
|Q|^{\frac1p}
\l(\frac1{|Q|}\int_{Q}|f(x)|^r\,dx\r)^{\frac1r}
\le
|Q|^{\frac1p}
\l(\frac1{|Q|}\int_{Q}|f(x)|^q\,dx\r)^{\frac1q}
\end{equation}
and hence
$$
\|f\|_{\cM^p_q(\R^n)}
\ge
\|f\|_{\cM^p_r(\R^n)}.
$$
This tells us that
\begin{equation}\label{1.2}
L^p(\R^n)=\cM^p_p(\R^n)
\subset\cM^p_q(\R^n)
\subset\cM^p_r(\R^n)
\text{ when }
p\ge q>r>0.
\end{equation}
If we let $f(x)=|x|^{-n/p}$, 
then the cube 
$R=(-t/2,t/2)^n$, $t>0$, 
attains its Morrey-norm. 
In fact, if $0<q<p<\8$, then  
$$
\sup_{Q\in\cQ}
|Q|^{\frac1p-\frac1q}
\l(\int_{Q}\frac1{|x|^{nq/p}}\,dx\r)^{\frac1q}
\le
|R|^{\frac1p-\frac1q}
\l(\int_{R}\frac1{|x|^{nq/p}}\,dx\r)^{\frac1q}
=
O\l((n(1-q/p))^{-1/q}\r)
$$
and that $f$ belongs to $\cM^p_q(\R^n)$. 
Because then $f$ does not belong to $L^p(\R^n)$, 
we see that 
the Morrey space $\cM^p_q(\R^n)$
is properly wider than 
the Lebesgue space $L^p(\R^n)$. 
The completeness of Morrey spaces 
follows easily by that of Lebesgue spaces. 

If the sequence of nonnegative functions 
$\{f_k\}_{k=1}^{\8} \subset\cM^p_q(\R^n)$
satisfies 
$f_k(x)\uparrow f(x)$, (a.e. $x\in\R^n$), 
then we have 
\begin{equation}\label{1.3}
\|f_k\|_{\cM^p_q(\R^n)}
\uparrow
\|f\|_{\cM^p_q(\R^n)}
\end{equation}
from the definition of the Morrey norm 
$\|\cdot\|_{\cM^p_q(\R^n)}$.
However, 
the following property is different from that of Lebesgue spaces. 

For any measurable set $E\subset\R^n$ 
such that $|E|<\8$ and 
any $f\in L^p(\R^n)$, 
we have by H\"{o}lder's inequality 
$$
\l|\int_{E}f(x)\,dx\r|
\le
\int_{E}|f(x)|\,dx
=
\int_{\R^n}\chi_{E}(x)|f(x)|\,dx
\le
|E|^{\frac1{p'}}
\|f\|_{L^p(\R^n)}<\8,
$$
where $p'$ is the conjugate number defined by $1/p+1/p'=1$ and 
$\chi_{E}$ stands for the characteristic function of $E$. 
While, if 
$f\in\cM^p_q(\R^n)$
then it follows from  the definition of Morrey-norm that 
\begin{align*}
\int_{Q}|f(x)|\,dx
&=
|Q|
\l(\frac1{|Q|}\int_{Q}|f(x)|\,dx\r)
\\ &\le
|Q|
\l(\frac1{|Q|}\int_{Q}|f(x)|^q\,dx\r)^{\frac1q}
\\ &=
|Q|^{1-\frac1p}
\cdot
|Q|^{\frac1p}
\l(\frac1{|Q|}\int_{Q}|f(x)|^q\,dx\r)^{\frac1q}
\\ &\le
\|f\|_{\cM^p_q(\R^n)}
|Q|^{\frac1{p'}}.
\end{align*}
This implies that for any family of 
counterable open cubes $\{Q_j\}_{j=1}^{\8}$ 
such that 
$E\subset\bigcup_jQ_j$, 
we have 
\begin{equation}\label{1.4}
\l|\int_{E}f(x)\,dx\r|
\le
\int_{E}|f(x)|\,dx
\le
\sum_j\int_{Q_j}|f(x)|\,dx
\le
\|f\|_{\cM^p_q(\R^n)}
\sum_j|Q_j|^{\frac1{p'}}.
\end{equation}
In general, 
for two real parameters 
$0<r\le\8$ and $0<d\le1$, 
the Hausdorff capacity or the Hausdorff content
of the set $E$ is defined by 
$$
H^d_r(E)
:=
\inf\sum_j|Q_j|^d,
$$
where the infimum is taken over all counterable cubes $\{Q_j\}_{j=1}^{\8}$ 
which cover $E$ with the side-length less than $r$.
Using this definition, 
we have by \eqref{1.4}
\begin{equation}\label{1.5}
\l|\int_{E}f(x)\,dx\r|
\le
H^{1/p'}_{\8}(E)
\|f\|_{\cM^p_q(\R^n)}.
\end{equation}
Of course, 
$|E|<\8$
does not always imply 
$H^{1/p'}_{\8}(E)<\8$. 
Thus, 
we cannot conclude  
from the fact that $|E|<\8$
that the left-hand side of this inequality 
is finite.

\subsection{Block spaces.} 
We shall define block spaces following \cite{BlRuVe}. 
Let $1<q\le p<\8$. 
We say that a function $b$ on $\R^n$ 
is a $(p',q')$-block provided that 
$b$ is supported on a cube $Q\in\cQ$ and 
satisfies 
\begin{equation}\label{1.6}
\l(\int_{Q}|b(x)|^{q'}\,dx\r)^{\frac1{q'}}
\le
|Q|^{\frac1p-\frac1q}.
\end{equation}
The space $\cB^{p'}_{q'}(\R^n)$ 
is defined by the set of all functions $f$ locally in $L^{q'}(\R^n)$ 
with the norm 
$$
\|f\|_{\cB^{p'}_{q'}(\R^n)}
:=
\inf\l\{
\|\{\lm_k\}\|_{l^1}:\,
f=\sum_k\lm_kb_k
\r\}<\8,
$$
where 
$\|\{\lm_k\}_{k=1}^{\8}\|_{l^1}
=
\sum_k|\lm_k|<\8$ 
and $b_k$ is a $(p',q')$-block, 
and the infimum is taken over all possible decompositions of $f$. 
By the definition of the norm 
we see that the inclusion 
\begin{equation}\label{1.7}
L^{p'}(\R^n)=\cB^{p'}_{p'}(\R^n)
\supset
\cB^{p'}_{q'}(\R^n)
\supset
\cB^{p'}_{r'}(\R^n)
\text{ when }
p\ge q>r>1.
\end{equation}
In \cite[Theorem 1]{BlRuVe} and \cite[Proposition 5]{Zo} 
the following was proved. 

\begin{proposition}\label{prop1.1}
Let $1<q\le p<\8$. Then 
the predual space of $\cM^p_q(\R^n)$ is 
$\cB^{p'}_{q'}(\R^n)$ 
in the following sense:

If $g\in\cM^p_q(\R^n)$, then 
$\int_{\R^n}f(x)g(x)\,dx$ 
is an element of 
$\cB^{p'}_{q'}(\R^n)^{*}$. 
Moreover, for any 
$L\in\cB^{p'}_{q'}(\R^n)^{*}$, 
there exists $g\in\cM^p_q(\R^n)$ 
such that 
$$
L(f)=\int_{\R^n}f(x)g(x)\,dx,
\quad
(f\in\cB^{p'}_{q'}(\R^n)).
$$
\end{proposition}
See \cite{Ka} for more details about the predual spaces.
We refer 
\cite{AdXi1,AdXi2,GoMu} 
for recent development of the theory of the predual spaces.

In this paper 
we shall prove the following theorem 
which assert the Fatou property of block spaces. 

\begin{theorem}\label{thm1.2}
Let $1<q\le p<\8$. Suppose that 
$f$ and $f_k$, $(k=1,2,\ldots)$, 
are nonnegative, 
$\|f_k\|_{\cB^{p'}_{q'}(\R^n)}\le 1$ 
and $f_k\uparrow f$ a.e. Then 
$f\in\cB^{p'}_{q'}(\R^n)$ and 
$\|f\|_{\cB^{p'}_{q'}(\R^n)}\le 1$. 
\end{theorem}
This quite simple fact can not be found in any literature as far as we know. 
Since the case when $p=q$ is clear from the Fatou lemma,
Theorem \ref{thm1.2} is significant only when $q<p$.
Seemingly, it is clear that we have $f \in \cB^{p'}_{q'}(\R^n)$.
However, it is difficult to find an expression of $f$.

The letter $C$ will be used for constants
that may change from one occurrence to another.
Constants with subscripts, such as $C_1$, $C_2$, do not change
in different occurrences.

\section{Proof of Theorem \ref{thm1.2}}\label{sec2}
In what follows 
we shall prove Theorem \ref{thm1.2}. 
We need the following lemmas. 

\begin{lemma}\label{lem2.1}
Let $1<q\le p<\8$. Then, 
a function $f$ belongs to $\cB^{p'}_{q'}(\R^n)$ 
if and only if 
there exists $g\in\cB^{p'}_{q'}(\R^n)$ 
such that 
$|f(x)|\le g(x)$, (a.e. $x\in\R^n$).
\end{lemma}

\begin{proof}
Suppose that $f\in\cB^{p'}_{q'}(\R^n)$. 
Then there exist a sequence $\{\lm_k\}_{k=1}^{\8}\in l^1$ 
and a $(p',q')$-block $b_k$ such that 
$f=\sum_k\lm_kb_k$. 
Letting $g=\sum_k|\lm_k||b_k|$, 
we have 
$g\in\cB^{p'}_{q'}(\R^n)$ and 
$|f|\le g$. 
Conversely, suppose that 
there exists $g\in\cB^{p'}_{q'}(\R^n)$ 
that satisfies $|f(x)|\le g(x)$. 
Decompose $g$ as 
$g=\sum_k\lm'_kb'_k$ 
where $\{\lm'_k\}_{k=1}^{\8}\in l^1$ and 
$b'_k$ is a $(p',q')$-block. 
Then we see that 
$$
\chi_{\{y:\,g(y)\neq 0\}}(x)
=
\sum_k\lm'_k\frac1{g(x)}b'_k(x)
$$
and, hence, 
$$
f(x)
=
\sum_k\lm'_k\frac{f(x)}{g(x)}b'_k(x).
$$
Since $|f(x)|/g(x) \le 1$, 
the function $(f(x)/g(x))b'_k(x)$ becomes 
a $(p',q')$-block. 
This proves the lemma.
\end{proof}

We denote by $\cD$ the family of all dyadic cubes of the form 
$Q=2^{-k}(i+[0,1)^n)$, 
$k\in\Z,\,i\in\Z^n$. 

\begin{lemma}\label{lem2.2}
Let $1<q\le p<\8$ and 
$f\in\cB^{p'}_{q'}(\R^n)$ with 
$\|f\|_{\cB^{p'}_{q'}(\R^n)}\le 1$. 
Then $f$ can be decomposed as 
$$
f=\sum_{Q\in\cD}\lm(Q)b(Q),
$$
where $\lm(Q)$ is a positive number with 
$$
\sum_{Q\in\cD}\lm(Q)\le 2\cdot3^n
$$
and $b(Q)$ is a $(p',q')$-block with 
$\supp b(Q)\subset 3Q$. 
\end{lemma}

\begin{proof}
First, decompose $f$ as 
$$
f=\sum_{k\in K}\lm_kb_k
$$
where $K\subset\N$ is an index set, 
$\sum_{k\in K}|\lm_k|\le 2$ and 
$b_k$ is a $(p',q')$-block. 
We divide $K$ into the disjoint sets 
$K(Q)\subset\N$, $Q\in\cD$, 
as 
$$
K=\bigcup_{Q\in\cD}K(Q)
$$
and $K(Q)$ fulfills 
$$
\supp b_k\subset 3Q
\text{ and }
|\supp b_k|\ge|Q|
\text{ when }
k\in K(Q).
$$
We now rewrite $f$ as 
\begin{align*}
f
&=
\sum_{k\in K}\lm_kb_k
=
\sum_{Q\in\cD}
\left(\sum_{k\in K(Q)}\lm_kb_k\right)
\\ &=
\sum_{Q\in\cD}
\l\{3^n\sum_{k\in K(Q)}|\lm_k|\r\}
\cdot
\l\{
\l(3^n\sum_{k\in K(Q)}|\lm_k|\r)^{-1}
\sum_{k\in K(Q)}\lm_kb_k
\r\}
=:
\sum_{Q\in\cD}\lm(Q)b(Q).
\end{align*}
It follows that 
$$
\sum_{Q\in\cD}\lm(Q)
=
3^n
\sum_{Q\in\cD}
\left(\sum_{k\in K(Q)}|\lm_k|\right)
=
3^n\sum_{k\in K}|\lm_k|
\le 2\cdot3^n
$$
and that
\begin{align*}
\lefteqn{
\l(3^n\sum_{k\in K(Q)}|\lm_k|\r)^{-1}
\l\|\sum_{k\in K(Q)}\lm_kb_k\r\|_{L^{q'}(\R^n)}
}\\ &\le
\l(3^n\sum_{k\in K(Q)}|\lm_k|\r)^{-1}
\sum_{k\in K(Q)}
|\lm_k|
\|b_k\|_{L^{q'}(\R^n)}
\\ &\le
\l(3^n\sum_{k\in K(Q)}|\lm_k|\r)^{-1}
|Q|^{\frac1p-\frac1q}
\sum_{k\in K(Q)}|\lm_k|
\le
|3Q|^{\frac1p-\frac1q},
\end{align*}
which implies that 
$b(Q)$ is a $(p',q')$-block with 
$\supp b(Q)\subset 3Q$. 
These complete the proof.
\end{proof}

\begin{proof}[Proof of Theorem \ref{thm1.2}]
We may assume that $1<q<p<\8$
as we remarked just below the statement.
By Lemma \ref{lem2.2} 
$f_k$ can be decomposed as 
$$
f_k
=
\sum_{Q\in\cD}\lm_k(Q)b_k(Q),
$$
where $\lm_k(Q)$ is a positive number with 
\begin{equation}\label{2.1}
\sum_{Q\in\cD}\lm_k(Q)\le 2\cdot3^n
\end{equation}
and $b_k(Q)$ is a $(p',q')$-block with 
$\supp b_k(Q)\subset 3Q$ and 
\begin{equation}\label{2.2}
\|b_k(Q)\|_{L^{q'}(\R^n)}
\le
|3Q|^{\frac1p-\frac1q}.
\end{equation}
Noticing \eqref{2.1}, \eqref{2.2} and 
the weak$*$-compactness of the Lebesgue space 
$L^{q'}(3Q)$, 
we now apply a diagonalization argument and, hence, 
we can select an infinite subsequence 
$\{f_{k_j}\}_{j=1}^{\8} \subset\{f_k\}_{k=1}^{\8}$ 
that satisfies the following: 
\begin{align}
&\nonumber
f_{k_j}
=
\sum_{Q\in\cD}\lm_{k_j}(Q)b_{k_j}(Q),
\\ &\label{2.3}
\lim_{j\to\8}\lm_{k_j}(Q)=\lm(Q),
\\ &\label{2.4}
\lim_{j\to\8}b_{k_j}(Q)=b(Q)
\text{ in the weak$*$-topology of }
L^{q'}(3Q),
\end{align}
where $b(Q)$ is a $(p',q')$-block with 
$\supp b(Q)\subset 3Q$. 
We set 
$$
f_0:=\sum_{Q\in\cD}\lm(Q)b(Q).
$$
Then, by the Fatou theorem and \eqref{2.1}, 
\begin{equation}\label{2.5}
\sum_{Q\in\cD}\lm(Q)
\le \liminf_{j\to\8}
\sum_{Q\in\cD}\lm_{k_j}(Q)
\le 2\cdot3^n,
\end{equation}
which implies 
$f_0\in\cB^{p'}_{q'}(\R^n)$.

We shall verify that 
\begin{equation}\label{2.6}
\lim_{j\to\8}
\int_{Q_0}f_{k_j}(x)\,dx
=
\int_{Q_0}f_0(x)\,dx
\end{equation}
for all $Q_0\in\cD$. 
Once \eqref{2.6} is established, 
we will see that $f=f_0$ and 
$f\in\cB^{p'}_{q'}(\R^n)$ 
by virtue of the Lebesgue differentiation theorem
because at least we know that 
$f_0$ locally in $L^{q'}(\R^n)$.

Let $\eps>0$ be given. 
We set 
$$
\l\{\begin{array}{l}
\cD_1(Q_0)
:=
\{Q\in\cD:\,
Q\cap Q_0\ne\0,\,
|3Q|\le c_1\},
\\
\cD_2(Q_0)
:=
\{Q\in\cD:\,
Q\cap Q_0\ne\0,\,
|3Q|\in(c_1,c_2)\},
\\
\cD_3(Q_0)
:=
\{Q\in\cD:\,
Q\cap Q_0\ne\0,\,
|3Q|\ge c_2\},
\end{array}\r.
$$
where we have defined, 
keeping in mind that $1/p-1/q<0$, 
$$
\l\{\begin{array}{l}
c_1^{\frac1p}
=
\dfrac{\eps}{12\cdot3^n},
\\
c_2^{\frac1p-\frac1q}
=
\dfrac{\eps}{12\cdot3^n|Q_0|^{1/q}},
\end{array}\r.
$$
It follows that 
\begin{align*}
\lefteqn{
\sum_{Q\in\cD_1(Q_0)}
\int_{Q_0}
\l|
\lm_{k_j}(Q)b_{k_j}(Q)(x)
-
\lm(Q)b(Q)(x)
\r|\,dx
}\\ &\le
\sum_{Q\in\cD_1(Q_0)}
(\lm_{k_j}(Q)+\lm(Q))|3Q|^{\frac1p}
\le
4\cdot3^nc_1^{\frac1p}
=\frac{\eps}3,
\end{align*}
where we have used 
\eqref{2.1}, \eqref{2.5} and 
$$
\int_{3Q}|b_{k_j}(Q)(x)|\,dx,\,
\int_{3Q}|b(Q)(x)|\,dx
\le
|3Q|^{\frac1p}
$$
(see \eqref{1.6}). 
It follows from H\"{o}lder's inequality that 
\begin{align*}
\lefteqn{
\sum_{Q\in\cD_3(Q_0)}
\int_{Q_0}
\l|
\lm_{k_j}(Q)b_{k_j}(Q)(x)
-
\lm(Q)b(Q)(x)
\r|\,dx
}\\ &\le
|Q_0|^{\frac1q}
\sum_{Q\in\cD_3(Q_0)}
(\lm_{k_j}(Q)+\lm(Q))
|3Q|^{\frac1p-\frac1q}
\le
4\cdot3^n
|Q_0|^{\frac1q}
c_2^{\frac1p-\frac1q}
=\frac{\eps}3,
\end{align*}
where we have used the fact that $1/p-1/q<0$ and 
$$
\|b_{k_j}(Q)\|_{L^{q'}(3Q)},\,
\|b(Q)\|_{L^{q'}(3Q)}
\le
|3Q|^{\frac1p-\frac1q}.
$$
Finally, 
\begin{align*}
\lefteqn{
\sum_{Q\in\cD_2(Q_0)}
\l|
\int_{Q_0}\lm_{k_j}(Q)b_{k_j}(Q)(x)\,dx
-
\int_{Q_0}\lm(Q)b(Q)(x)\,dx
\r|
}\\ &\le
\sum_{Q\in\cD_2(Q_0)}
\l|
\int_{Q_0}\lm_{k_j}(Q)b_{k_j}(Q)(x)\,dx
-
\int_{Q_0}\lm(Q)b_{k_j}(Q)(x)\,dx
\r|
\\ &\quad +
\sum_{Q\in\cD_2(Q_0)}
\l|
\int_{Q_0}\lm(Q)b_{k_j}(Q)(x)\,dx
-
\int_{Q_0}\lm(Q)b(Q)(x)\,dx
\r|.
\end{align*}
{}From \eqref{2.3} and 
the fact that $\cD_2(Q_0)$ contains the only finite number of dyadic cubes, 
\begin{eqnarray*}
&&\sum_{Q\in\cD_2(Q_0)}
\l|
\int_{Q_0}\lm_{k_j}(Q)b_{k_j}(Q)(x)\,dx
-
\int_{Q_0}\lm(Q)b_{k_j}(Q)(x)\,dx
\r|\\ 
&&\le
c_2^{\frac1p}
\sum_{Q\in\cD_2(Q_0)}
|\lm_{k_j}(Q)-\lm(Q)|
\le\frac{\eps}6
\end{eqnarray*}
for large $j$.
{}From \eqref{2.4},
\begin{align*}
\lefteqn{
\sum_{Q\in\cD_2(Q_0)}
\l|
\int_{Q_0}\lm(Q)b_{k_j}(Q)(x)\,dx
-
\int_{Q_0}\lm(Q)b(Q)(x)\,dx
\r|
}\\ &\le
2\cdot3^n
\sup_{Q\in\cD_2(Q_0)}
\l|
\int_{Q_0}b_{k_j}(Q)(x)\,dx
-
\int_{Q_0}b(Q)(x)\,dx
\r|
\le\frac{\eps}6
\end{align*}
for large $j$.
These prove \eqref{2.6}.

Since $f_k\uparrow f$ a.e., 
we must have by \eqref{2.6} 
$$
\int_{Q_0}f(x)\,dx
=
\int_{Q_0}f_0(x)\,dx
$$
for all $Q_0\in\cD$. 
This yields $f=f_0$ a.e., 
by the Lebesgue differential theorem, and, hence, 
$f\in\cB^{p'}_{q'}(\R^n)$.
Since we have verified 
$f\in\cB^{p'}_{q'}(\R^n)$, 
it follows from Proposition \ref{prop1.1} that
$$
\|f\|_{\cB^{p'}_{q'}(\R^n)}
=
\sup\l\{
\l|\int_{\R^n}f_k(x)g(x)\,dx\r|:\,
k=1,2,\ldots,\,
\|g\|_{\cM^p_q(\R^n)}\le 1
\r\}
\le 1.
$$
This completes the proof of the theorem.
\end{proof}

\section{Banach function spaces}\label{sec3}
To state our next result, 
we need terminology from the theory of the Banach function spaces 
introduced in the book \cite{BeSh}. 
We place ourselves in the setting of a $\sg$-finite measure space $(R,\mu)$. 
Let $\M^{+}$ be the cone of all $\mu$-measurable functions on $R$ 
assuming their values lie in $[0,\8]$. 

\begin{definition}\label{def3.1}
A mapping $\rho:\M^{+}\to[0,\8]$ 
is called a \lq\lq Banach function norm" 
if, 
for all $f,g,f_n,(n=1,2,3,\ldots),$ 
in $\M^{+}$, 
for all constants $a\ge 0$ and 
for all $\mu$-measurable subsets $E$ of $R$, 
the following properties hold:

\begin{enumerate}
\item[(P1)]
$\rho(f)=0\,\Leftrightarrow\,f=0$ 
$\mu$-a.e.; 
$\rho(af)=a\rho(f)$;
$\rho(f+g)\le\rho(f)+\rho(g)$;

\item[(P2)]
$0\le g\le f\text{ $\mu$-a.e. }
\,\Rightarrow\,
\rho(g)\le\rho(f)$;

\item[(P3)]
$0\le f_n\uparrow f\text{ $\mu$-a.e. }
\,\Rightarrow\,
\rho(f_n)\uparrow\rho(f)$;

\item[(P4)]
$\mu(E)<\8
\,\Rightarrow\,
\rho(\chi_{E})<\8$;

\item[(P5)]
$\mu(E)<\8
\,\Rightarrow\,
\int_{E}fd\mu\le C_{E}\rho(f)$
for some constant $C_{E}$, 
$0<C_{E}<\8$, 
depending on $E$ and $\rho$ but independent of $f$.
\end{enumerate}
\end{definition}
We say that 
a weaker version of the Fatou property (property P3) 
is \lq\lq the Fatou property".

Let $\M$ denote the collection of all extended scalar-valued 
(real or complex) $\mu$-measurable functions on $R$. 
As usual, 
any two functions coinciding $\mu$-a.e. shall be identified. 

\begin{definition}\label{def3.2}
Let $\rho$ be a function norm. 
The collection $X=X(\rho)$ 
of all functions $f$ in $\M$ 
for which $\rho(|f|)<\8$ is called 
a \lq\lq Banach function space". 
For each $f\in X$, define 
$$
\|f\|_{X}:=\rho(|f|). 
$$
\end{definition}

Let $1<q<p<\8$. 
The Morrey space $\cM^p_q(\R^n)$ and 
the block space $\cB^{p'}_{q'}(\R^n)$ 
are not Banach function spaces, 
since the norm 
$\|\cdot\|_{\cM^p_q(\R^n)}$
fails property P5 and the norm 
$\|\cdot\|_{\cB^{p'}_{q'}(\R^n)}$
fails property P4 (cf. \eqref{1.5}).

\begin{example}\label{exm3.3}
Here, 
we exhibit an example showning that 
$\|\cdot\|_{\cM^p_q(\R^n)}$
fails property P5 
when $1<q<p<\8$.
For simplicity, 
we let $n=1$ and $1<q<2=p$; 
other cases are dealt analogously.
Let us consider the sequence
$$
(a_1,a_2,\ldots)
=
\l(1,\frac14,\frac14,\frac1{16},\frac{1}{16},\frac{1}{16},\frac{1}{16},\ldots\r),
$$
that is, 
$a_i$ is a decreasing sequence and 
$4^{-l}$ appears $2^l$ times for 
$l=0,2,\ldots$.

Let $\al(p,q)\gg 1$. We define 
$$
E
:=
\bigcup_{j=1}^{\8}
(\al(p,q)^j,\al(p,q)^j+a_j).
$$
Then $|E|=2$. Define 
$$
f(t)
:=
\sum_{j=1}^{\8}
4^{j/p}\chi_{(\al(p,q)^j,\al(p,q)^j+a_j)}(t),
\quad
(t\in\R).
$$
Then $f$ belongs to $\cM^p_q(\R)$. 
Meanwhile, 
$$
\int_{E}f(t)\,dt
=
\sum_{j=1}^{\8}
4^{j/p-j}\cdot2^j=\8.
$$
\end{example}

\begin{definition}\label{def3.4}
If $\rho$ is a function norm, 
its \lq\lq associate norm" $\rho'$ is defined on $\M^{+}$ by 
\begin{equation}\label{3.1}
\rho'(g)
:=
\sup\l\{
\int_{R}fg\,d\mu:\,
f\in\M^{+},\rho(f)\le 1\r\},
\quad
(g\in\M^{+}).
\end{equation}
\end{definition}

We have the following property:

\begin{theorem}[{\rm\cite[Theorem 2.2]{BeSh}}]\label{thm3.5}
Let $\rho$ be a function norm. Then 
the associate norm $\rho'$ is itself a function norm.
\end{theorem}

\begin{definition}\label{def3.6}
Let $\rho$ be a function norm and 
let $X=X(\rho)$ be the Banach function space determined by $\rho$ 
as in Definition \ref{def3.2}. 
Let $\rho'$ be the associate norm of $\rho$. 
The Banach function space $X(\rho')$ determined by $\rho'$ 
is called the \lq\lq associate space" 
of $X$ and is denoted by $X'$.
\end{definition}

\begin{theorem}[{\rm\cite[Theorem 2.7]{BeSh}}]\label{thm3.7}
Every Banach function space $X$ 
coincides with its second associate space $X''$. 
In other words, 
a function $f$ belongs to $X$ 
if and only if 
it belongs to $X''$, 
and in that case
$$
\|f\|_{X}=\|f\|_{X''},
\quad
(f\in X=X'').
$$
\end{theorem}

\section{Application of Theorem \ref{thm1.2}}\label{sec4}
In what follows 
we shall apply Theorem \ref{thm1.2} and 
characterize the predual of block spaces. 

\begin{theorem}\label{thm4.1}
Let $1<q\le p<\8$. Then 
the associate space $\cM^p_q(\R^n)'$
coincides with 
the block space $\cB^{p'}_{q'}(\R^n)$.
\end{theorem}

\begin{proof}
We see that 
$\cB^{p'}_{q'}(\R^n)\subset\cM^p_q(\R^n)'$
by Proposition \ref{prop1.1}. 
So we shall verify the converse. 
Suppose that a measurable function $f$ satisfies 
\begin{equation}\label{4.1}
\sup\l\{
\l|\int_{\R^n}f(x)g(x)\,dx\r|:\,
\|g\|_{\cM^p_q(\R^n)}\le 1
\r\}\le 1.
\end{equation}
Then we first see that 
$|f(x)|<\8$, (a.e. $x\in\R^n$). 
Splitting $f$ into its real and imaginary parts and 
each of these into its positive and negative parts,
we may assume without loss of generality that $f\ge 0$. 
For $k=1,2,\ldots$, 
set $Q_k:=(-k,k)^n$ and let 
$$
f_k(x):=\min(f(x),k)\chi_{Q_k}(x)
\quad (x \in \R^n).
$$
We notice that 
$f_k\in\cB^{p'}_{q'}(\R^n)$ and 
$\|f_k\|_{\cB^{p'}_{q'}(\R^n)}\le 1$
by Lemma \ref{lem2.1}, 
Proposition \ref{1.1} and 
\eqref{4.1}. 
Since $f_k\uparrow f$ a.e., 
it follows from Theorem \ref{1.2} that 
$f\in\cB^{p'}_{q'}(\R^n)$ and 
$\|f\|_{\cB^{p'}_{q'}(\R^n)}\le 1$.
This proves the theorem.
\end{proof}

By Theorems \ref{thm3.7} and \ref{thm4.1}
(or directly Proposition \ref{prop1.1}),
one sees that 
$$
\cB^{p'}_{q'}(\R^n)'
=
\cM^p_q(\R^n)''
=
\cM^p_q(\R^n).
$$
Furthermore, 
from the fact that 
$\cM^p_q(\R^n)''=\cM^p_q(\R^n)$,
we are able to characterize the predual of block spaces 
following the argument in \cite{BeSh}.

\begin{definition}\label{def4.2}
Let $1<q\le p<\8$. 
The closure in $\cM^p_q(\R^n)$ of the set of 
all finite linear combination of the characteristic functions of sets of finite measure 
is denoted by 
$\wt{\cM}^p_q(\R^n)$. 
\end{definition}

\begin{theorem}\label{thm4.3}
Let $1<q\le p<\8$. Then 
the predual space of 
$\cB^{p'}_{q'}(\R^n)$ is 
$\wt{\cM}^p_q(\R^n)$ 
in the following sense:

If $g\in\cB^{p'}_{q'}(\R^n)$, then 
$\int_{\R^n}f(x)g(x)\,dx$ 
is an element of 
$\wt{\cM}^p_q(\R^n)^{*}$. 
Moreover, for any 
$L\in\wt{\cM}^p_q(\R^n)^{*}$, 
there exists $g\in\cB^{p'}_{q'}(\R^n)$ 
such that 
$$
L(f)=\int_{\R^n}f(x)g(x)\,dx,
\quad
(f\in\wt{\cM}^p_q(\R^n)).
$$
\end{theorem}

\begin{proof}
The first assertion is clear. 
So we shall prove that 
$\wt{\cM}^p_q(\R^n)^{*}\subset\cB^{p'}_{q'}(\R^n)$. 
Thanks to Theorem \ref{thm4.1}, 
we need only show 
$$
\wt{\cM}^p_q(\R^n)^{*}
\subset
\cM^p_q(\R^n)'.
$$
Suppose that $L$ belongs to 
$\wt{\cM}^p_q(\R^n)^{*}$.
We shall exhibit a function $g$ in 
$\cM^p_q(\R^n)'$ 
such that we have
\begin{equation}\label{4.2}
L(f)=\int_{\R^n}f(x)g(x)\,dx,
\quad
(f\in\wt{\cM}^p_q(\R^n)).
\end{equation}

If we let 
$Q_i:=i+[0,1)^n$, $i\in\Z^n$, 
then the sequence 
$\{Q_i\}_{i\in\Z^n}$ forms 
disjoint subsets of $\R^n$, 
each of which has measure one and 
whose union is all of $\R^n$. 
For each $i\in\Z^n$, 
let $\cA_i$ denote the Lebesgue measurable subsets of $Q_i$ 
and define a set-function $\lm_i$ on $\cA_i$ by 
$$
\lm_i(A)=L(\chi_{A}),
\quad
(A\in\cA_i).
$$
Notice that 
$\lm_i(A)$ is well-defined for all $A\in\cA_i$ 
because $\chi_{A}$ belongs to 
$\wt{\cM}^p_q(\R^n)$. 

We claim that $\lm_i$ is countably additive on $\cA_i$. 
Indeed, 
let $(A_k)_{k=1}^{\8}$ 
be a sequence of disjoint sets 
from $\cA_i$ and let 
$$
B_l=\bigcup_{k=1}^lA_k,
\quad
(l=1,2,\ldots),
\qquad
A=\bigcup_{k=1}^{\8}A_k
=
\bigcup_{l=1}^{\8}B_l.
$$
It follows from \eqref{1.2} and 
the Lebesgue dominated convergence theorem 
that 
$$
\|\chi_{A}-\chi_{B_l}\|_{\cM^p_q(\R^n)}
\le
\|\chi_{A}-\chi_{B_l}\|_{L^p(Q_i)}
\to 0
\text{ as }l\to\8.
$$
The continuity and linearity of $L$ give 
$$
\lm_i(A)
=
L(\chi_{A})
=
\lim_{l\to\8}L(\chi_{B_l})
=
\lim_{l\to\8}
\sum_{k=1}^l
L(\chi_{A_k})
=
\sum_{k=1}^{\8}\lm_i(A_k),
$$
which establishes the claim. 

Since 
$|\lm_i(A)|\le\|L\|_{\wt{\cM}^p_q(\R^n)^{*}}$ 
for all $A\in\cA_i$ and 
$\lm_i(A)=0$ for all 
$A\in\cA_i$ such that $|A|=0$, 
by the Radon-Nikodym theorem, 
there is a unique measurable function 
$g_i$ on $Q_i$ such that 
$$
L(\chi_{A})=\lm_i(A)
=
\int_{\R^n}\chi_{A}(x)g_i(x)\,dx,
\quad
(A\in\cA_i).
$$
Since the sets $Q_i$ are disjoint 
we may define a function $g$ on all of $\R^n$ 
by setting $g=g_i$ on each $Q_i$. 
Clearly, 
\begin{equation}\label{4.3}
L(\chi_{E})
=
\int_{\R^n}\chi_{E}(x)g(x)\,dx
\end{equation}
for all characteristic functions of sets of finite measure $\chi_{E}$.

We first show that 
$g$ belongs to $\cM^p_q(\R^n)'$. 
Choose and fix $f$ in $\cM^p_q(\R^n)$. 
Let, for $l=1,2,\ldots$, 
$$
f_l(x)
:=
\sum_{k=1}^{4^l}
\frac{k}{2^l}\chi_{F_{k,l}}(x),
$$
where 
$$
F_{k,l}
:=
\l\{x\in\R^n:\,
|x|<2^l,\,
\frac{k}{2^l}\le|f(x)|<\frac{k+1}{2^l}
\r\}.
$$
If we suppose for the moment that 
$g$ is real-valued, then 
$f_l\cdot\sgn(g)$ becomes 
a finite linear combination of characteristic functions of sets of finite measure. 
Hence, we may apply \eqref{4.3} and use the linearity of $L$ to obtain
$$
\int_{\R^n}f_l(x)|g(x)|\,dx
=
L(f_l\cdot\sgn(g))
\le
\|L\|_{\wt{\cM}^p_q(\R^n)^{*}}
\|f_l\|_{\cM^p_q(\R^n)}.
$$
Letting $l\to\8$, we have 
$$
\int_{\R^n}|f(x)g(x)|\,dx
\le
\|L\|_{\wt{\cM}^p_q(\R^n)^{*}}
\|f\|_{\cM^p_q(\R^n)}
$$
from the monotone convergence theorem and 
the Fatou property of Morrey norm.
This means that 
$g$ belongs to $\cM^p_q(\R^n)'$.
If $g$ is complex-valued, 
then the same argument applied separately to the real and imaginary parts of $g$ 
shows that each of these is 
in $\cM^p_q(\R^n)'$ and hence that 
$g$ again belongs to $\cM^p_q(\R^n)'$.

Write, 
for a function $f$ which can be written as
a finite linear combination of characteristic functions of sets of finite measure, 
$$
L(f)
=
\int_{\R^n}f(x)g(x)\,dx
$$
and observe the continuity of both sides 
on $\cM^p_q(\R^n)$. Then 
we conclude that \eqref{4.2} holds.
This complete the proof of the theorem.
\end{proof}

\begin{remark}\label{rem4.4}
Let $1<q\le p<\8$. 
Let $\cC_0$ be the class of continuous functions 
with compact support in $\R^n$. 
The Zorko space $\cZ^p_q(\R^n)$ 
is defined by the closure in 
$\cM^p_q(\R^n)$ of $\cC_0$. 
In \cite{AdXi2}, 
Adams and Xiao pointed out 
(without detailed proof) 
$\cZ^p_q(\R^n)$ 
is the predual of 
$\cB^{p'}_{q'}(\R^n)$. 
In \cite{IzSaYa}, 
Izumi, Sato and Yabuta 
gave a detailed proof of this fact on the unit circle. 
The idea used in the proof of Theorem \ref{thm1.2} comes from their nice paper. 
\end{remark}

We shall use $\{E_k\}_{k=1}^{\8}$ 
to denote an arbitrary sequence of measurable subsets of $\R^n$. 
We shall write $E_k\to\0$ a.e., if 
the characteristic functions $\chi_{E_k}$ 
converge to $0$ pointwise a.e.
Notice that the sets $E_k$ are not required to have finite measure.

\begin{definition}\label{def4.5}
Let $1<q\le p<\8$. A function $f$ 
in $\cM^p_q(\R^n)$ is said to have 
\lq\lq absolutely continuous norm" 
in $\cM^p_q(\R^n)$ if 
$\|f\chi_{E_k}\|_{\cM^p_q(\R^n)}\to 0$ 
for every sequence 
$\{E_k\}_{k=1}^{\8}$ 
satisfying 
$E_k\to\0$ a.e. 
The set of all functions 
in $\cM^p_q(\R^n)$ of absolutely continuous norm 
is denoted by 
$\wht{\cM}^p_q(\R^n)$.
\end{definition}

\begin{theorem}\label{thm4.6}
Let $1<q\le p<\8$. Then 
$$
\wt{\cM}^p_q(\R^n)
=
\wht{\cM}^p_q(\R^n).
$$
\end{theorem}

\begin{proof}
By \cite[Theorem 3.13]{BeSh}, 
we need only verify that 
the characteristic function $\chi_{E}$ has absolutely continuous norm 
for every set $E$ of finite measure.
Let $\{F_k\}_{k=1}^{\8}$ 
be an arbitrary sequence for which 
$F_k\to\0$ a.e. Then 
it follows from \eqref{1.2} and 
the Lebesgue dominated convergence theorem 
that 
$$
\|\chi_{E}\chi_{F_k}\|_{\cM^p_q(\R^n)}
\le
\|\chi_{E}\chi_{F_k}\|_{L^p(\R^n)}
\to 0
\text{ as }k\to\8,
$$
which proves the theorem.
\end{proof}

\section{Miscellaneous}\label{sec5}

\begin{example}\label{exm5.1}
Let $1<q<p<\8$. We show that 
$\wt{\cM}^p_q(\R^n)$ and 
$\cM^p_q(\R^n)$ are different spaces. 
In fact, 
the former is narrower than the latter.
We prove this by giving an example when $n=1$; 
other cases can be dealt similarly.

Set 
$$
E
:=
\bigcup_{k=1}^{\8}
(k-1+k^{\frac{p}{p-q}},\,k+k^{\frac{p}{p-q}}).
$$
Then we see that 
$\chi_{E}$ belongs to 
$\cM^p_q(\R^n)$
but does not belong to 
$\wt{\cM}^p_q(\R^n)$.
\end{example}

\begin{example}\label{exm5.2}
Let $1<q<p<\8$ and 
$L:\,\cM^p_q(\R^n)\to\R$ 
be a bounded linear functional.
Then in view of the embedding 
$L^p(\R^n)\hookrightarrow\cM^p_q(\R^n)$,
one has a function 
$g\in L^{p'}(\R^n)$ 
such that
$$
L(f)=\int_{\R^n}f(x)g(x)\,dx,
\quad
(f\in L^p(\R^n)).
$$
However, 
it can happen that $L$ is not zero 
even when $g\equiv 0$;
One can show this by an example.

Recall the set $E$ defined in 
Example \ref{exm5.1}. Set, 
for $k=1,2,\ldots$, 
$$
I_k
:=
(k-1+k^{\frac{p}{p-q}},\,k+k^{\frac{p}{p-q}}).
$$
Then, 
$$
\lim_{k\to\8}
\int_{I_k}\chi_{E}(x)\,dx=1.
$$
With this in mind, 
let us define a closed subspace $H$ by
$$
H
:=
\l\{f\in\cM^p_q(\R):\,
\lim_{k\to\8}
\int_{I_k}f(x)\,dx
\text{ exists }
\r\}.
$$
Then, 
from the definition of the norm, 
we have 
$$
\lim_{k\to\8}
\l|\int_{I_k}f(x)\,dx\r|
\le
\|f\|_{\cM^p_q(\R)}.
$$
Consequently, 
it follows from the Hahn-Banach theorem that 
the mapping
$$
f\in H
\mapsto 
\lim_{k\to\8}\int_{I_k}f(x)\,dx
\in\R
$$
extends to a continuous linear functional $L$.
Observe that 
$L(\chi_{E})=1$ 
and hence 
$L\ne 0$.
Meanwhile, 
$L$ annihilates any compactly supported function in 
$\cM^p_q(\R)$ 
because such a function belongs to $H$.
Therefore, if one considers a function $g$ satisfying
$$
L(f)=\int_{\R^n}f(x)g(x)\,dx
$$
for all $f \in L^p(\R^n)$,
then one obtains 
$g\equiv 0$ 
by virtue of the Lebesgue dominated convergence theorem. 
\end{example}

We end this paper with the following proposition. 

\begin{proposition}\label{prop5.3}
Let $1<q\le p<\8$. Suppose that 
$f\in L^{q'}(\R^n)$ 
has compact support. Then
there exists a finite sequence
$\{\lm_j\}_{j=1}^{N}$
of nonnegative numbers and
$\{b_j\}_{j=1}^{N}$
of $(p',q')$-blocks such that 
$$
f=\sum_{j=1}^{N}\lm_jb_j
\text{ and }
\sum_{j=1}^{N}\lm_j
\le 8
\|f\|_{\cB^{p'}_{q'}(\R^n)}.
$$
\end{proposition}

\begin{proof}
The proof will be complete once we show that
there exists a finite sequence
$\{\lm_j\}_{j=1}^{N}$
of nonnegative numbers and 
$\{b_j\}_{j=1}^{N}$
of $(p',q')$-blocks such that
$$
f=\sum_{j=1}^{N}\lm_jb_j
\text{ and }
\sum_{j=1}^{N}\lm_j
\le 2
\|f\|_{\cB^{p'}_{q'}(\R^n)}
$$
when $f$ is positive.

We know that, 
as is illustrated by the proof of Lemma \ref{lem2.1}, 
there exist an infinite sequence 
$\{\Lm_j\}_{j=1}^{\8}$
of nonnegative numbers and  an infinite sequence
$\{B_j\}_{j=1}^{\8}$
of nonnegative $(p',q')$-blocks 
such that
$$
f=\sum_{j=1}^{\8}\Lm_jB_j
\text{ and }
\sum_{j=1}^{\8}\Lm_j
\le
\frac32\|f\|_{\cB^{p'}_{q'}(\R^n)}.
$$
Suppose that the support of $f$ is engulfed by a large cube $Q_0$.
By using the characteristic functions, 
we may as well assume that $B_j$ is supported on $Q_0$.
Then we have
$$
f=\sum_{j=1}^{N-1}\Lm_jB_j
+
\sum_{j=N}^{\8}\Lm_jB_j
$$
and
\begin{align*}
\sum_{j=1}^{N-1}\Lm_j
+
|Q_0|^{\frac1q-\frac1p}
\l\|\sum_{j=N}^{\8}\Lm_jB_j\r\|_{L^{q'}(\R^n)}
\le
\frac32\|f\|_{\cB^{p'}_{q'}(\R^n)}
+
|Q_0|^{\frac1q-\frac1p}
\l\|f-\sum_{j=1}^{N-1}\Lm_jB_j\r\|_{L^{q'}(\R^n)}.
\end{align*}
By the monotone convergence theorem,
we see that
$$
\sum_{j=1}^{N-1}\Lm_j
+
|Q_0|^{\frac1q-\frac1p}
\l\|\sum_{j=N}^{\8}\Lm_jB_j\r\|_{L^{q'}(\R^n)}
\le 
2\|f\|_{\cB^{p'}_{q'}(\R^n)}
$$
as long as $N$ is sufficient large.
Thus, if we define 
$\lm_1,\lm_2,\ldots,\lm_N$
and
$b_1,b_2,\ldots,b_N$ 
as
$$
\lm_j=\Lm_j,\,b_j=B_j 
\text{ when }j=1,2,\ldots,N-1,
$$
and
$$
\Lm_N
=
|Q_0|^{\frac1q-\frac1p}
\l\|\sum_{j=N}^{\8}\Lm_jB_j\r\|_{L^{q'}(\R^n)},
\quad
b_N
=
\begin{cases}
\dfrac{1}{\lm_N}
\sum_{j=N}^{\8}\Lm_jB_j
&\,\lm_N\ne 0
\\
0&\,\lm_N=0,
\end{cases}
$$
then we obtain the desired decomposition.
\end{proof}

\end{document}